\documentclass[11pt]{amsart}
\usepackage{color}

\usepackage[left=3.5cm,top=3.5cm,right=3.5cm,bottom=3.5cm]{geometry}

\usepackage[utf8]{inputenc}

\usepackage{mathabx}

\newtheorem{thm}{Theorem}
\newtheorem{col}{Corollary}
\newtheorem{lem}{Lemma}

\theoremstyle{remark}

\def\R{\mathbb{R}}

\def\N{\mathbb{N}}
\def\E{\mathbb{E}}

\def\ilsk#1#2{\left<#1 , #2\right>}

\def\1{1{\hskip -2.5 pt}\hbox{\textup{l}}}

\begin{document}

\title{Trait evolution in two--sex populations}
\author[P. Zwole\'nski]{Pawe{\l} Zwole{\'n}ski}
\address{Institute of Mathematics,
Polish Academy of Sciences,
Bankowa 14, 40-007 Katowice, Poland.}
\email{pawel.zwolenski@gmail.com}
\thanks{This research was partially supported
by Warsaw Center of Mathematics and Computer Science from the KNOW grant of the Polish Ministry of Science and Higher Education.}
\begin{abstract}
We present an individual--based model of phenotypic trait evolution in two--sex populations, which includes semi--random mating of individuals of the opposite sex, natural death and intra--specific competition. By passing the number of individuals to infinity, we derive the macroscopic system of nonlinear differential equations describing the evolution of trait distributions in male and female subpopulations. We study solutions, give criteria for persistence or extinction, and state a theorem on asymptotic stability, which we later apply to particular examples of trait inheritance.
\end{abstract}
\maketitle

\section{Introduction}
Over last few decades two--sex populations were often studied from the viewpoint of mathematical modelling  (see e.g. \cite{asm,had}). One of the first attempted descriptions of birth rate and matting functions in two--sex populations were given in \cite{good, ken}. In the paper \cite{fred} two assumptions on mating function were stated and its general form was provided. The above--mentioned models exhibit exponential grow due to the lack of competition. Addition of intrasexual competition to the models leads to logistic equations and results in more realistic bounded solutions and usually causes stabilization of population size (see  e.g. \cite{liu, rose}). One of the reason for the interest in more complicated two--sex models comes from attempts to describe sexually transmitted diseases in human population (see \cite{d1, d2}). Modern models use more and more advanced mathematical tools, such as partial differential equations or stochastic processes, in order to include some of the complicated structures of populations.  For example, attempted description of age--structured two--sex populations can lead to systems of non--linear partial differential equations (see e.g. \cite{bus, au}).

This paper studies constant lifetime phenotypic trait evolution in two--sex populations, using techniques of individual--based modeling. In spite of the vast  literature concerning this type of models in evolutionary biology and population dynamics, a great deal of them describe asexual populations (see \cite{ch, ferriere_tran, fo_mel}). Only few models concerning hermaphroditic organisms have appeared so far (see \cite{colet, rud_zwo}). To the extend of our knowledge, there is a lack in the field of individual--based modeling for two--sex populations and mathematical analysis of equations derived as macroscopic approximations. 

In this paper a phenotypic trait is not sex--linked, i.e., all the trait--coding genes lie outside the sex chromosome (allosome). We distinguish two subpopulations of males and females of the same species, and assume that the mating is semi--random, i.e., there is a function of individual capability of mating, which depends on individual's trait and sex. This function is a rate at which every individual mates with a random partner, which is chosen from all living individuals of the opposite sex according to some distribution also based on the capability function (see \cite{arino, rud_w, rud_w2, rud_zwo}).  After mating an offspring is born, and its sex is male or female with probability $\frac{1}{2}$ by virtue of Fisher's principle on sex ratio. The phenotypic trait is inherited from parents as a mean parental trait with some stochastic noise. Moreover, individuals can die naturally or in intra--specific competition at trait--dependent rates. All of the above events happen randomly in discrete population in continuous time. The evolution of population is described by a sequence of measure--valued stochastic processes (individual--based model).

The main goal of the paper is to study macroscopic equations, which are derived as a law of large numbers for the stochastic processes considered, when the number of individuals tends to infinity. We obtain the system of two nonlinear differential equations, which describe the evolution of trait distributions in male and female subpopulations, respectively. We study existence and uniqueness of the solutions, and examine total number of individuals -- we give criterion for persistence of population, and also show when extinction occurs. The most important result concerns asymptotic stability of solutions: we investigate when the distributions of phenotypic traits in male and female subpopulations tend to a stationary solution, which is the same for both sexes. This implies that the distribution of non--sex--linked phenotypic traits become the same for males and females after long time of evolution. The asymptotic result is a transmission of analogous theorem from hermaphroditic populations (see \cite{rud_zwo}). In order to show applications of our result, we give two examples of mean parental trait inheritance with different forms of stochastic noise.

The structure of the paper is the following: in the next section we explain notation and gather all the model assumptions. Also in that part, we describe the individual--based model and introduce stochastic processes, discuss their existence and state the limit theorem (the law of large numbers), which gives the equation for macroscopic approximation of studied stochastic processes. In Section 3, we derive and study equations for trait distributions in male and female subpopulations. We prove existence and uniqueness of solutions in space of finite Borel measures and study total number of populations. Criteria for persistence and extinction are given. In Section 4, theorem on asymptotic stability of trait distributions is stated and proven. Section 5 includes examples of trait inheritance and application of our asymptotic results. In Section 6 we summarize our work and give future perspectives for extending the model.

\section{The model}
\subsection{Assumptions and parameters of model}
We assign every individual some element from a set $\mathbb{X}:=X\times\{\Venus,\Mars\}$, where $X$ is non--empty subset of $\R^d$ for some $d\in\N$. The first coordinate of $\mathbb X$'s members describes individual's \textit{phenotypic trait}, the second one its \textit{sex} ($\Venus$ corresponds to female and $\Mars$ is male). Both phenotypic trait and sex are assumed to be constant in individual's lifetime. We simply call elements of $\mathbb X$ \textit{traits}. For convenience sake we also consider following sets $\mathbb F:=X\times\{\Venus\}$, $\mathbb M:=X\times\{\Mars\}$. We also impose the following useful convention: if $\mathbf x\in\mathbb X$ is individual's trait, then we denote by the same, but non--bold letter $x$ its phenotypic trait, i.e., $x$ is the first coordinate of $\mathbf x$.

\subsubsection{Mating}
We adapt \textit{semi--random mating/coagulation} models (see \cite{colet, rud_w, rud_w2, rud_zwo}) to the two--sex population case: an individual of trait $\mathbf x \in\mathbb X$ has a rate $p(\textbf x)$ of \textit{initial capability of mating}, at which it starts mating by choosing a partner of the opposite sex. A female (resp. male) chooses a partner of trait $\mathbf y$ form all living males (resp. females) from the distribution
$$\frac{p(\textbf y)}{\sum_i p(\textbf w_i)}$$
where the sum in the denominator extends over all living males (resp. females), and $\textbf w_i$ are their traits. In other words, if the trait distribution in current population is described by some measure $\mu$, then \textit{mating rate} $m(\mathbf x, \mathbf y, \mu)$ of individuals of traits $\textbf x,\mathbf y$ is 
\begin{equation}\label{mrate}
m(\mathbf x,\mathbf y,\mu)=p(\mathbf x)\1_{\mathbb F}(\mathbf x) \frac{p(\mathbf y)\1_{\mathbb M}(\mathbf y)}{\int_\mathbb{M} p(\mathbf w)\mu(d\mathbf w)} +p(\mathbf x)\1_{\mathbb M}(\mathbf x) \frac{p(\mathbf y)\1_{\mathbb F}(\mathbf y)}{\int_\mathbb{F} p(\mathbf w)\mu(d\mathbf w)},
\end{equation}
where $\1_A$ is an indicator function of set $A$. We assume that the function $p$ is positive, continuous and upper--bounded by some constant $\overline p>0$.

\subsubsection{Trait inheritance}
We assume that after every mating a new individual is born, and according to Fisher's principle (see \cite{fisher}), is male or female with probability $\frac{1}{2}$. We consider only these phenotypic traits which are not sex--linked. We assume that if $x,y\in X$  are parental phenotypic traits, then the offspring's trait $z$ comes from distribution $k(x,y,dz)$. We assume that for every $x,y\in X$ the measure $k(x,y,\cdot)$ is a Borel probability measure and for every bounded and continuous function $f\colon X\to\R$, the mapping $(x,y)\mapsto \int_X f(z)k(x,y,dz)$ is continuous. Moreover, we suppose that $k(x,y,\cdot)=k(y,x,\cdot)$ for every $x,y\in X$. Denote by $K(x,y,\cdot)$ measure on $\mathbb X$ satisfying $K(x,y,A\times\{\Venus\})=K(x,y,A\times\{\Mars\})=\frac{1}{2}k(x,y,A)$ for every Borel subset $A\subset X$.

\subsubsection{Natural death and competition}
We assume that if an individual has trait $\mathbf x$, then it can \textit{die naturally} at rate $D(\mathbf x)$. Moreover, we consider \textit{intra--specific competition} (see \cite{fo_mel, rud_zwo}): in current population described by a measure $\mu$, an individual with trait $\mathbf x$ dies in competition with rate $C(\mathbf x,\mu):=\int_\mathbb X U(\mathbf x, \mathbf y)\mu(d\mathbf y).$ Function $U(\mathbf x, \mathbf y)$ describes ``how often'' individual with trait $\mathbf x$ looses competition with one of trait $\mathbf y$ (\textit{competition kernel}).  We assume that functions $D$ and  $U$ are positive, continuous and upper--bounded by $\overline D, \overline U$ respectively.

\subsubsection{Population dynamics}
We consider a finite population in continuous time. At a random time individuals with traits $\mathbf x, \mathbf y$ mate with rate $m(\mathbf x, \mathbf y, \mu)$ given by the formula (\ref{mrate}), where $\mu=\sum_{i=1}^n \delta_{\mathbf w_i}$, provided  the current population consists of individuals with traits $\mathbf w_1,\ldots,\mathbf w_n$. After mating an offspring is born with probability $1$. Its sex is male or female with probability $\frac{1}{2}$ and phenotypic trait comes from distribution $k(x,y,\cdot)$. Moreover, an individual of trait $\mathbf x$ can die naturally at rate $D(\mathbf x)$ or loosing competition with other members of population at rate $U(\mathbf x,\mu)$.
All the events and interactions are assumed to take place independently.

\subsection{Stochastic processes and limit theorem}
Denote by $\mathcal M(\mathbb X)$ set of all finite Borel measures on $\mathbb X$. For any $N\in\N$ we define following sequence of subsets of $\mathcal M(\mathbb X)$
\begin{equation}
\mathcal M^N=\Bigg\{\frac{1}{N}\sum_{i=1}^n \delta_{\mathbf w_i}\colon n\in\N, \mathbf w_i\in\mathbb X\Bigg\}.
\end{equation}
We study a sequence $(\nu^N)_{N\in\N}$ of $\mathcal M^N$--valued, continuous time stochastic processes given by the infinitesimal generators
\begin{multline}\label{LN}
L^N\phi(\nu)=N\int_\mathbb X\int_\mathbb X\int_\mathbb X\bigg(\phi\Big(\nu+\frac{1}{N}\delta_{\mathbf z}\Big)-\phi(\nu)\bigg)m(\mathbf x,\mathbf y,\nu)K(x,y,d\mathbf z)\nu(d\mathbf x)\nu(d\mathbf y)\\
+N\int_\mathbb X\bigg(\phi\Big(\nu-\frac{1}{N}\delta_{\mathbf x}\Big)-\phi(\nu)\bigg)\Big(D(\mathbf x)+\frac{1}{N}C(\mathbf x,\nu)\Big)\nu(d\mathbf x)
\end{multline}
for any measurable and bounded $\phi\colon\mathbb X\to\R$.
Notice that the processes given by above generators are jump processes on $\mathcal M^N$. The first term on the right--hand side of (\ref{LN}) describes mating and trait inheritance including Fisher's principle. The second term on the right--hand side of (\ref{LN}) corresponds to natural death and competition, whose rate is rescaled by the factor $\frac{1}{N}$.  Given initial value $\nu^N_0\in\mathcal M^N$, under the model's assumptions, there exists a $\mathcal M^N$--valued Markov process $(\nu^N_t)_{t\geq 0}$ with infinitesimal generator given by (\ref{LN}) (see \cite{fo_mel}).
\begin{thm}\label{pwl}
Suppose that the sequence of initial values $(\nu^N_0)_{N\in\N}$ converges to some measure $\nu\in \mathcal M(\mathbb{X})$ in topology of weak convergence of measures. Then for every $T>0$ the sequence of processes converges in distribution in Skorokhod space $\mathcal D\big([0,T], \mathcal M(\mathbb{X})\big)$ to the deterministic and continuous flow of measures $\mu\colon [0,T]\to\mathcal M(\mathbb{X})$ satisfying for every $0\leq t\leq T$ following equation
\begin{multline}\label{eq_weak}
\ilsk{\mu_t}{\phi}=\ilsk{\mu_0}{\phi}+\int_0^t\int_\mathbb X\int_\mathbb X\int_\mathbb X \phi(\mathbf z)m(\mathbf x,\mathbf y,\nu)K(x,y,d\mathbf z)\mu_s(d\mathbf x)\mu_s(d\mathbf y)ds\\
-\int_0^t\int_\mathbb X\phi(x)\bigg(D(\mathbf x)+\int_\mathbb X U(\mathbf x,\mathbf y)\mu_s(d\mathbf y)\bigg)\mu_s(d\mathbf x)ds
\end{multline}
for any measurable and bounded $\phi\colon\mathbb X\to\R.$
\end{thm}
The proof of above statement is standard and can be adapted e.g. from \cite{rud_w}.

\section{System of macroscopic equations}
\subsection{The system and solutions} Define $m_t(A)=\mu_t(A\cap \{\Mars\})$ and $f_t(A)=\mu_t(A\cap \{\Venus\})$ -- the measures describing evolution of traits in male and female subpopulations, respectively. Setting in (\ref{eq_weak}) $\phi=\1_{A\times \{\Mars\}}$ and then $\phi=\1_{A\times \{\Venus\}}$, and rewriting into differential form, we obtain
\begin{equation}\label{systemzero}
\left\{\begin{array}{l}
\begin{aligned}
\frac{d}{dt}m_t(dz)&=\int_X\int_X p(x,y,m_t,f_t) k(x,y,dz)f_t(dx)m_t(dy)\\
&-\bigg(D_m(z)+\int_X U_{m,m}(z,y)m_t(dy)+\int_X U_{m,f}(z,y)f_t(dy)\bigg)m_t(dz),
\end{aligned}\vspace*{0.5cm}\\
\begin{aligned}
\frac{d}{dt}f_t(dz)&=\int_X\int_X p(x,y,m_t,f_t) k(x,y,dz)f_t(dx)m_t(dy)\\
&-\bigg(D_f(z)+\int_X U_{f,m}(z,y)m_t(dy)+\int_X U_{f,m}(z,y)f_t(dy)\bigg)f_t(dz),
\end{aligned}\end{array}\right.
\end{equation}
where 
\begin{enumerate}
\item $p(x,y,\mu,\nu):=\frac{1}{2}p_f(x)p_m(y)\Big(\frac{1}{\int_X p_m(w)\mu(dw)}+\frac{1}{\int_X p_f(w)\nu(dw)}\Big),$
\item $p_f(x):=p\big((x,\Venus)\big), p_m(x):=p\big((x,\Mars)\big)$ -- individual capabilities of mating in the female and respectively, male subpopulations, 
\item $D_f(x):=D\big((x,\Venus)\big), D_m(x):=D\big((x,\Mars)\big)$ -- natural death rates in the female and respectively, male subpopulations,
\item $U_{f,f}(x,y):=U\big((x,\Venus),(y,\Venus)\big)$ -- competition kernel between females,
\item $U_{m,m}(x,y):=U\big((x,\Mars),(y,\Mars)\big)$ -- competition kernel between males,
\item $U_{m,f}(x,y):=U\big((x,\Mars),(y,\Venus)\big)$ -- competition kernel describing the rate of competition loss of the males due to females,
\item $U_{f,m}(x,y):=U\big((x,\Venus),(y,\Mars)\big)$ -- competition kernel describing the rate of competition loss of the females due to males.
\end{enumerate}
Note that system (\ref{systemzero}) is a trait--structured population analogue of some classic two--sex population models well--known from the literature. For example, our model contains more general, trait--dependent version of some of the mating functions studied in \cite{fred, good, ken}, and intrasexual competition considered in \cite{liu, rose}. 

\medskip

Recall that for any finite Borel measure $\mu$ on $X$ we can introduce the total variation norm by formula $\|\mu\|_{\textup{TV}}:=\sup_f \int_X f(x)\mu(dx),$ where the supremum is taken over the set of all measurable functions $f\colon X\to\R$ such that $|f(x)|\leq 1$ for all $x\in X$. Under weaker conditions than in Theorem \ref{pwl}, we are able to prove existence and uniqueness of solutions to (\ref{systemzero}) in stronger norm than Theorem~\ref{pwl} provides. The proof of following result can be easily adapted from proof of Theorem~2 in \cite{rud_zwo}.
\begin{thm}\label{exist1}
Suppose that $p_f,p_m,D_f,D_m,U_{f,f},U_{f,m},U_{m,f},U_{m,m}$ are measurable, upper--bounded and bounded from below by some positive constants. For every $\mu_0,\nu_0\in\mathcal M(X)$ there exists unique pair of functions $\mu,\nu\colon [0,\infty)\to\mathcal M(X),$ which is the solution of system $(\ref{systemzero})$ with initial conditions $\mu_0,\nu_0$. The functions $\mu_t,\nu_t$ are continuous and bounded function in the norm $\|\cdot\|_{\textup{TV}}.$
\end{thm}

If we additionally suppose that for every $x,y\in X$ there exists a density $\kappa(x,y,\cdot)\in L^1$ of measure $k(x,y,dz)$ with respect to Lebesgue measure, then using Radon--Nikodym theorem, we can also prove the following theorem.

\begin{thm}
Under the assumptions of Theorem~\ref{exist1}, if $\mu_0,\nu_0$ have densities $u_0,v_0\in L^1$ with respect to Lebesgue measure, than for every $t\geq 0$ solutions $\mu_t,\nu_t$ of $(\ref{systemzero})$ have densities $u(t,\cdot),v(t,\cdot)\in L^1$. Functions $u(t,z),v(t,z)$ are unique solutions of the following system
\begin{equation*}
\left\{\begin{array}{l}
\begin{aligned}
\frac{\partial }{\partial t}u(t,z)&=\int_X\int_X p\Big(x,y,u(t,\xi)d\xi,v(t,\xi)d\xi\Big) \kappa(x,y,z)u(t,x)v(t,y)dxdy\\
&-\bigg(D_m(z)+\int_X U_{m,m}(z,y)u(t,y)dy+\int_X U_{m,f}(z,y)v(t,y)dy\bigg)u(t,z),
\end{aligned}\vspace*{0.5cm} \\
\begin{aligned}
\frac{\partial }{\partial t}v(t,z)&=\int_X\int_X p\Big(x,y,u(t,\xi)d\xi,v(t,\xi)d\xi\Big) \kappa(x,y,z)u(t,x)v(t,y)dxdy\\
&-\bigg(D_f(z)+\int_X U_{f,m}(z,y)u(t,y)dy+\int_X U_{f,m}(z,y)v(t,y)dy\bigg)v(t,z).
\end{aligned}\end{array}\right.
\end{equation*}
\end{thm}

\subsection{Total number of individuals}
In this chapter we consider a case when all of the rates $D_m,$ $D_f,$  $U_{m,m},$ $U_{m,f},$ $U_{f,m},$ $U_{f,f}$ are constant and positive. We assume that  $p_m,$ $p_f$ are constant, non--negative and $p_m+p_f>0$. Then the system (\ref{systemzero}) has the following form
\begin{equation}\label{sysconst}
\left\{\begin{array}{l}
\begin{aligned}
\frac{d}{dt}m_t(dz)&=\frac{p_fF(t)+p_mM(t)}{2M(t)F(t)}\int_X\int_X k(x,y,dz)f_t(dx)m_t(dy)\\
&-\Big(D_m+U_{m,m}M(t)+U_{m,f}F(t)\Big)m_t(dz),
\end{aligned}\vspace*{0.5cm} \\
\begin{aligned}
\frac{d}{dt}f_t(dz)&=\frac{p_fF(t)+p_mM(t)}{2M(t)F(t)}\int_X\int_X k(x,y,dz)m_t(dx)f_t(dy)\\
&-\Big(D_f+U_{f,m}M(t)+U_{f,f}F(t)\Big)f_t(dz),
\end{aligned}\end{array}\right.
\end{equation}
 where $M(t):=m_t(X)$ and $F(t):=f_t(X)$.
In order to investigate asymptotic properties of total numbers of individuals in male and female subpopulations, let us denote $\lambda(t):=\big(p_f F(t)+p_m M(t)\big)/2.$
Integrating both sides of equations of the system (\ref{sysconst}), we obtain
\begin{equation}\label{systotal}
\left\{\begin{array}{l}
M'(t)=\lambda(t) - \big(D_m + U_{m,m} M(t)+U_{m,f} F(t)\big)M(t), \vspace*{0.2cm} \\
F'(t)=\lambda(t) - \big(D_f + U_{f,m} M(t)+U_{f,f} F(t)\big)F(t).
\end{array}\right.
\end{equation}
The function $\lambda(t)$ can be interpreted as a birth rate of each subpopulations at time~$t$. This rate is identical in males and females due to the Fisher's principle. The form of $\lambda(t)$ in our model generalizes some of the birth rates studied before in literature. For instance some of the birth rates investigated in \cite{fred, good, ken} can be obtained simply by taking $p_m=p_f=1$ and $p_m=1, p_f=0$. However, since our equations contain additional terms, which are responsible for competition, the solutions exhibit asymptotic properties different than exponential growth of above models.  Similar intersexual competition terms were studied in \cite{rose} Section 4. A particular case of system (\ref{systotal}) appeared also in \cite{liu} and the existence of globally asymptotically stable stationary solution was obtained. We provide a result concerning the same type of asymptotic behavior, but in more general setting.

\begin{thm}\label{astotal} Consider solution $\big(M(t),F(t)\big)$ of system $(\ref{systotal})$ with initial condition $(M_0,F_0),$ $M_0,F_0>0$.
Suppose that the inequality
\begin{equation}\label{persis}
\frac{p_m}{D_m}+\frac{p_f}{D_f}>2
\end{equation}
holds. Then there exist unique $\bar M,\bar  F>0$ such that $\big(M(t),F(t)\big)\to (\bar  M, \bar  F)$ exponentially, as $t\to\infty$.
On the contrary, if the inequality
\begin{equation}\label{extinct}
\frac{p_m}{D_m}+\frac{p_f}{D_f}\leq 2,
\end{equation}
holds,  then $\big(M(t),F(t)\big)\to (0,0)$ as $t\to\infty$.
\end{thm}
\begin{proof}
Without loss of generality we can assume that $p_f>0$. Suppose that condition (\ref{persis}) holds.
  We show that the following system of polynomial equations
\begin{equation}\label{poly}
\left\{\begin{array}{l}p_mM+p_fF - 2D_mM - 2U_{m,m} M^2-2U_{m,f} FM=0, \vspace*{0.2cm} \\
p_mM+p_fF- 2D_fF - 2U_{f,m} MF-2U_{f,f} F^2=0.\end{array}\right. 
\end{equation}
has unique non--trivial solution $(\bar M,\bar F)$. Denote by $h_1$ and $h_2$ curves given by first and second equation, respectively. Notice, that if $p_m- 2D_m = \frac{U_{m,m}p_f}{U_{m,f}},$ or $p_f-2D_f=\frac{U_{f,f}p_m}{U_{f,m}},$ the system has only one positive solution, which is easy to find explicitly. 
Denote the implicit formulas for both curves $F_1(M)=\frac{p_mM-2D_mM-2U_{m,m}M^2}{2U_{m,f}M-p_f}$ and $M_2(F)=\frac{p_fF-2D_fF-2U_{f,f}F^2}{2U_{f,m}F-p_m}$ and let $A=\frac{p_f}{2U_{m,f}}$, $B=\frac{p_m}{2U_{f,m}}.$ Then
\begin{enumerate}
\item if $p_m- 2D_m > \frac{U_{m,m}p_f}{U_{m,f}},$ and $p_f-2D_f>\frac{U_{f,f}p_m}{U_{f,m}},$ then
$$\lim_{M\nearrow A} F_1(M)=-\infty,\, \lim_{M\searrow A} F_1(M)=+\infty,\, \lim_{F\nearrow B} M_2(F)=-\infty,\, \lim_{F\searrow B}M_2(F)=+\infty,$$
\item if $p_m- 2D_m > \frac{U_{m,m}p_f}{U_{m,f}},$ and $p_f-2D_f<\frac{U_{f,f}p_m}{U_{f,m}},$ then
$$\lim_{M\nearrow A} F_1(M)=-\infty,\, \lim_{M\searrow A} F_1(M)=+\infty,\, \lim_{F\nearrow B} M_2(F)=+\infty,\, \lim_{F\searrow B}M_2(F)=-\infty,$$
\item if $p_m- 2D_m < \frac{U_{m,m}p_f}{U_{m,f}},$ and $p_f-2D_f>\frac{U_{f,f}p_m}{U_{f,m}},$ then
$$\lim_{M\nearrow A} F_1(M)=+\infty,\, \lim_{M\searrow A} F_1(M)=-\infty,\, \lim_{F\nearrow B} M_2(F)=-\infty,\, \lim_{F\searrow B}M_2(F)=+\infty,$$
\item if $p_m- 2D_m < \frac{U_{m,m}p_f}{U_{m,f}},$ and $p_f-2D_f<\frac{U_{f,f}p_m}{U_{f,m}},$ then
$$\lim_{M\nearrow A} F_1(M)=+\infty,\, \lim_{M\searrow A} F_1(M)=-\infty,\, \lim_{F\nearrow B} M_2(F)=+\infty,\, \lim_{F\searrow B}M_2(F)=-\infty.$$
\end{enumerate} 
Moreover one can check that $F=-\frac{U_{m,m}}{U_{m,f}}M+a,$ $M=-\frac{U_{f,f}}{U_{f,m}}F+b$, for some constants $a,b\in\R$ give the formulas for asymptotes of $h_1$ and $h_2$, respectively.  Since in all of the above cases each of $h_1$ and $h_2$ has two different asymptotes, there must be that $h_1$ and $h_2$ are hyperbolas or two straight lines. Since $F=\frac{2D_m-p_m}{p_f} M$, $F=\frac{p_m}{2D_f-p_f} M$ are tangent lines at $(0,0)$ to $h_1$ and $h_2$, respectively, from (\ref{persis}) one can prove that 
$$\frac{p_m}{2D_f-p_f}>\frac{2D_m-p_m}{p_f} ,$$
i.e., $h_2$ lies above $h_1$ in some right neighborhood of $0$ in the following sense: $F_1(M)< F_2(M)$ for small $M>0$ (here $F_2(M)$ is an explicit formula for $F$ derived from the second equation of (\ref{poly})). Now it is easy to check that in every case such two hyperbolas (or two pairs of straight lines) must cross at exactly one positive point (by virtue of Darboux property of continuous functions). 

Consider solution $\big(M(t),F(t)\big)$ of (\ref{systotal}) and let $\rho(M,F)=-1/(MF)$. For every $t>0$
$$\frac{\partial}{\partial M} \big(\rho H_1\big)\big(M(t),F(t)\big)+\frac{\partial}{\partial F} \big(\rho H_2\big)\big(M(t),F(t)\big)\geq \frac{U_{m,m}}{F(t)}+\frac{U_{f,f}}{M(t)}>0,$$
where $H_1$ and $H_2$ are functions from right--hand sides of the first and second equation of system (\ref{systotal}), respectively.
Since solutions of the system are upper--bounded, from Dullac--Bendixon theorem any solution $\big(N(t),F(t)\big)$ tends to one of the stationary points, as $t\to\infty$. In order to finish the proof, we show that $(0,0)$ is retracting.
Denote $\beta(t)=p_m M(t)/D_m+p_f F(t)/D_f$. From (\ref{systotal}) we obtain
\begin{equation}\label{eqtotalhelp}
\beta'(t)=\bigg(\frac{p_m}{D_m}+\frac{p_f}{D_f}-2\bigg)\lambda(t)-C_1p_m^2M^2(t)-2C_2p_mp_fM(t)F(t)-C_3p_f^2F^2(t),
\end{equation}
where $C_1,C_2,C_3>0$ are some constants. Let $C=\max\{C_1,C_2,C_3\}.$ Then
$$
\beta'(t)\geq \lambda(t)\bigg(\frac{p_m}{D_m}+\frac{p_f}{D_f}-2-4C\lambda(t)\bigg).
$$
If $N(t),F(t)$ are small enough that $0<\lambda(t)<\frac{p_m}{4CD_m}+\frac{p_f}{4CD_f}-\frac{1}{2C},$ then $\beta'(t)>0,$ and at least one of the functions $N(t),F(t)$ grows strictly. Thus $(0,0)$ is retracting point.

Suppose that the opposite case holds, i.e., the inequality (\ref{extinct}) is satisfied. Notice that from (\ref{poly}) it follows that
$$\bigg(\frac{p_m}{D_m}+\frac{p_f}{D_f}-2\bigg)\lambda=A M^2+B MF+C F^2,$$
where $\lambda=p_mM/2+p_fF/2$ and $A,B,C>0$ are some constants. Since left--hand side of above is non--positive and right--hand side is non--negative, there is no positive solution, and consequently $M=F=0$ is the only solution to the system.
Moreover, from  (\ref{eqtotalhelp}) it follows that $\beta'(t)\leq -C\beta(t),$ where $C>0$ is some constant. Then $\beta(t)\to 0$ exponentially, and thus  $\big(N(t),F(t)\big)\to (0,0),$ as $t\to\infty$.
\end{proof}

Suppose that (\ref{persis}) holds.
Substituting $m_t(dz)=M(t)\mu_t(dz)$ and $f_t(dz)=F(t)\nu_t(dz)$ in (\ref{sysconst}) and time scaling $t\mapsto \int_0^t \frac{\lambda(s)}{M(s)}ds$ lead to following system
\begin{equation}\label{sysprobmes}
\left\{\begin{array}{l}
\begin{aligned}
\frac{d}{dt}\mu_t(dz)+\mu_t(dz) =\int_X\int_X k(x,y,dz)\mu_t(dx)\nu_t(dy),
\end{aligned}\vspace*{0.5cm} \\
\begin{aligned}
\frac{d}{dt}\nu_t(dz)+A(t)\nu_t(dz) =A(t)\int_X\int_X k(x,y,dz)\mu_t(dx)\nu_t(dy),
\end{aligned}
\end{array}\right.
\end{equation}
where $A(t)=M(t)/F(t)$. Notice, that if $\mu_0,\nu_0$ are probability measures, then for every $t>0$, solutions $\mu_t,\nu_t$ are probability measures as well. We denote by $\mathcal M_{\textrm{Prob}}$ the space of all Borel probability measures on $X$.

\section{Asymptotic stability of trait distribution}
\subsection{Statement of the main result} In this section we assume that $X$ is closed interval in $\R$.
From Theorem~\ref{astotal}, under condition (\ref{persis}), male and female subpopulations sizes stabilize on positive level, i.e., $A(t)\to A$ for some $A>0$ as $t\to\infty.$ The number $A$ is a ratio of male subgroup size to size of female subgroup in stable population. Denote $\mathcal P(\mu,\nu)(dz)=\int_X\int_X k(x,y,dz)\mu(dx)(dy)$. In this section, we study long time behavior of the following system
\begin{equation}\label{sysprobmesconst}
\left\{\begin{array}{l}
\begin{aligned}
\mu_t'+\mu_t =\mathcal P(\mu_t,\nu_t),
\end{aligned}\vspace*{0.2cm} \\
\begin{aligned}
\nu_t'+A\,\nu_t=A\,\mathcal P(\mu_t,\nu_t),
\end{aligned}
\end{array}\right.
\end{equation}
where $A>1$ is some constant. Later on, we compare solutions of initial system with solutions of (\ref{sysprobmesconst}) in order to obtain asymptotic behavior also for solutions of (\ref{sysprobmes}). 

We assume that 
\begin{equation}\label{ksr1}
\int_X |z|k(x,y,dz)\leq a_1+a_2|x|+a_3|y|,
\end{equation}
for some constants $a_1,a_2,a_3>0$, and
\begin{equation}\label{sredniaK}
\int_X zk(x,y,dz)=\frac{x+y}{2}.
\end{equation}
The above condition is a reasonable biological assumption and means that the expected offspring's trait is a mean parental trait.

For any $\gamma\geq1$ and $\alpha\leq\beta$ we introduce 
$\mathcal M_\gamma:=\big\{\mu\in\mathcal M_{\textrm{Prob}}\colon \int_X |x|^\gamma\mu(dx)<\infty\big\}$ and $\mathcal M_{[\alpha,\beta]}:=\big\{\mu\in \mathcal M_1\colon \alpha\leq \int_X x\mu(dx)\leq \beta \big\}.$
For any two measures $\mu,\nu\in \mathcal{M}_1$, we define the \textit{Wasserstein distance} by the formula
\begin{equation}
d(\mu,\nu)=\sup_{f\in\textup{Lip}_1} \int_X f(z)(\mu-\nu)(dz),
\end{equation}
where $\textup{Lip}_1$ is a set of all continuous functions $f\colon X\to \R$ such that $|f(x)-f(y)|\leq |x-y|,$ for any $x,y\in X$ .

Denote by $\mathcal K(x,y,\cdot)$ the cumulative distribution function of measure $k(x,y,\cdot)$, i.e., $\mathcal K(x,y,z)=k\big(x,y,X\cap (-\infty,z]\big).$
The main theorem of the paper is
\begin{thm}\label{mainthm} Fix $\alpha,\beta\in X$, $\alpha\leq \beta$. Suppose that
\begin{enumerate}
\item[(i)] for all $y,z\in X$ the function $\mathcal K(x,y,z)$ is absolutely continuous with respect to $x$ and for every $a,b,y\in X$
$$
\int_X \Big|\frac{\partial}{\partial x}\mathcal K(a,y,z)-\frac{\partial}{\partial x}\mathcal K(b,y,z)\Big|\,dz< 1,
$$
\item[(ii)] there exist constants $\gamma>1$, $C>0$ and $L<1$ such that
$$
\int_X |x|^{\gamma}\mathcal{P}(\mu,\nu)(dx)\le C+L\max\bigg\{\int_X |x|^{\gamma}\mu(dx),\int_X |x|^{\gamma}\nu(dx)\bigg\} .
$$ 
for every $\mu,\nu\in\mathcal M_{\gamma}\cap\mathcal M_{[\alpha,\beta]}$.
\end{enumerate}
Then for every $\mu_0,\nu_0\in\mathcal M_{[\alpha,\beta]}$ there exists unique solution $\mu,\nu\colon [0,\infty)\to\mathcal M_{[\alpha,\beta]}$ of system $(\ref{sysprobmesconst})$ with initial values $\mu_0,\nu_0$. Moreover, there exists unique measure $\mu^*\in\mathcal M_{[\alpha,\beta]}$ such that $\mathcal P(\mu^*,\mu^*)=\mu^*$, and for every initial measures $\mu_0,\nu_0\in\mathcal M_{[\alpha,\beta]}$ corresponding functions $\mu_t,\nu_t$ converge to $\mu^*$ in space $(\mathcal M_{[\alpha,\beta]},d)$, as $t\to\infty$.
\end{thm}

The concept of the above theorem and idea of its proof come from similar result for hermaphroditic populations (see \cite{rud_zwo}). Nonetheless, the proofs differ in many details and contain some nontrivial and new elements. For the convenience of the reader, we present the full reasoning.

\subsection{Convergence of measures} In order to investigate asymptotic properties of the solutions, we recall some basic theory concerning convergence of measures. We start with method of computing Wasserstein distance, which can be found in \cite{rud_zwo} as Lemma~1.

\begin{lem}
\label{waslem} 
The  Wasserstein distance between measures
 $\mu,\nu\in \mathcal{M}_{1}$ can be computed by the formula
\begin{equation}
d(\mu,\nu)=\int_X |\Phi(x)|\,dx,
\end{equation}
where $\Phi(z)=(\mu-\nu)\left(X\cap(-\infty,z]\right)$ is a cumulative distribution function of the signed measure $\mu-\nu$.
\end{lem}

Consider probability measures $\mu$ and $\mu_n$, $n\in\N$, on the set $X$. We recall that the sequence $\mu_n$ converges \textit{weakly} (or \textit{in a weak sens}) to $\mu$, if for any continuous and bounded function $f\colon X\to\R$
$$\int_X f(x)\,\mu_n(dx)\to\int_X f(x)\,\mu(dx),$$
as $n\to\infty$.
It is well--known that the convergence in Wasserstein distance implies weak convergence of measures. Moreover, the space of probability Borel measures on any complete metric space is also a complete metric space with the Wasserstein distance (see e.g. \cite{bolley, rach}). The convergence of a sequence $\mu_n$ to $\mu$ in the space $\mathcal{M}_{1,q}:=\big\{\mu\in\mathcal M_1\colon  \int_X x\mu(dx)=q\big\}$ is equivalent to the following condition (see \cite{villani}, Definition 6.7 and Theorem 6.8)
\begin{equation}\tag{C}\label{warrown}
\mu_n\to\mu \mbox{ weakly, as } n\to\infty \quad \mbox{ and } \quad\lim_{R\to\infty} \limsup_{n\to\infty} \int_{X_R} |x|\mu_n(dx)=0,
\end{equation}
where $X_R:=\{x\in X\colon |x|\geq R \}.$ Fix $q\in X$, $\alpha>1$, and $m>0$. 

\begin{lem}\label{relcomp}
Assume that $\alpha,\beta\in X$, $\alpha\leq\beta$, $m>0$ and $\gamma>1$.  Consider the family 
$$\mathcal M_{[\alpha,\beta],\gamma, m}:=\bigg\{\mu\in \mathcal M_1\colon \alpha\leq \int_X x\mu(dx)\leq \beta, \, \int_X |x|^\gamma\mu(dx)\leq m\bigg\}.$$
Then $\mathcal M_{[\alpha,\beta],\gamma, m}$ is relatively compact subset of $(\mathcal M_{[\alpha,\beta]},d).$
\end{lem}
\begin{proof}
Fix any sequence $(\mu^n)$ of measures from $\mathcal M_{\alpha,\beta,\gamma, m}$. Since $\alpha\leq \E\mu\leq \beta$ there exists a subsequence $(p_n)$ such that $\lim_n \E\mu^{p_n}= E$ for some $\alpha\leq E \leq \beta.$ Consider sequence of measures $(\bar \mu^{p_n})$ given by 
$$\bar\mu^{p_n}=\left\{\begin{array}{l l}
a_n \mu^{p_n}+(1-a_n)\delta_\beta, & \textrm{ if }\, \E\mu^{p_n}\leq E,\\
a_n \mu^{p_n}+(1-a_n)\delta_\alpha, & \textrm{ if }\,\E\mu^{p_n}> E,
\end{array}\right.$$
where $a_n\in[0,1]$ satisfies $a_n\E\mu^{p_n}+(1-a_n)\beta =E$ (resp. $a_n\E\mu^{p_n}+(1-a_n)\alpha =E$). From $\lim_n \E\mu^{p_n}= E$, it follows that $a_n\to 1$ and $\lim_n d(\mu^{p_n},\bar \mu^{p_n})=0,$ and moreover
$$\int_X |x|^\gamma \bar\mu^{p_n}(dx)\leq a_n\int_X |x|^\gamma\mu^{p_n}(dx)+(1-a_n)\Big(|\alpha|^\gamma+|\beta|^\gamma\Big)\leq m+|\alpha|^\gamma+|\beta|^\gamma<\infty.$$
Consequently, there exist $\mu^*\in\mathcal M_{1,E}$ and some subsequence $(q_n)$ of $(p_n)$ such that $\lim_n d(\bar \mu^{q_n},\mu^*)=0$ (the condition (C) is satisfied, see remarks after Lemma 1 in \cite{rud_zwo}). Finally,
$$d(\mu^{q_n},\mu^*)\leq d(\mu^{q_n},\bar \mu^{q_n})+d(\bar\mu^{q_n},\mu^*)\to 0,$$
as $n\to\infty$.
\end{proof}

\subsection{Proof of the main result} We split the proof of Theorem \ref{mainthm} into a sequence of lemmas.
\begin{lem}\label{lemma1}
Assume that condition \textup{(i)} of Theorem $\ref{mainthm}$ is satisfied. Then
\begin{equation}\label{kontrakcja}
d\big(\mathcal P(\mu^1,\nu^1), \mathcal P(\mu^2,\nu^2)\big)< \max\big\{d(\nu^1, \nu^2), d(\mu^1, \mu^2)\big\}
\end{equation}
for any $\mu^1,\mu^2,\nu^1,\nu^2\in\mathcal M_1$ such that $\int_X x \mu^1(dx)=\int_X x\mu^2(dx)$ and $\int_X x\nu^1(dx)=\int_X x\nu^2(dx).$
\end{lem}
\begin{proof}
Denote by $\Phi$ the cumulative distribution function of signed measure $\mu_1-\mu_2$, and consider $\Phi^+(x):=\max\{0,\Phi(x)\}$ and $\Phi^-(x):=\max\{0,-\Phi(x)\}.$
The assumption $\int_X x \mu^1(dx)=\int_X x\mu^2(dx)$ implies
$$\int_X \Phi^+(x)dx=\int_X \Phi^-(x)dx=\frac{1}{2}\int_X |\Phi(x)|dx.$$
Since $\Phi^+$ and $\Phi^-$ are non--negative and have the same integral condition (i) implies
$$\int_X\bigg|\int_X \frac{\partial}{\partial x}\mathcal K(x,y,z)\Phi^+(x)dx - \int_X \frac{\partial}{\partial x}\mathcal K(x,y,z)\Phi^-(x)dx \bigg|dz<\int_X \Phi^+(x)dx.$$
Integrating by parts we obtain
$$\int_X\mathcal K(x,y,z)\Phi(dx)=-\int_X \frac{\partial}{\partial x}\mathcal K(x,y,z)\Phi(x)dx$$
and consequently by Lemma \ref{waslem}
\begin{equation}
\int_X\bigg|\int_X\mathcal K(x,y,z)\Phi(dx) \bigg|dz<\frac{1}{2}\int_X |\Phi(x)|dx=\frac{1}{2}d(\mu^1,\mu^2).
\end{equation}
In the same way we prove that if $\Psi$ is the cumulative distribution function of signed measure $\nu^1-\nu^2$, then
\begin{equation}
\int_X\bigg|\int_X\mathcal K(x,y,z)\Psi(dy) \bigg|dz<\frac{1}{2}\int_X |\Psi(y)|dy=\frac{1}{2}d(\nu^1,\nu^2).
\end{equation}
Now, since
\begin{multline*}
\mathcal P(\mu^1,\nu^1)- \mathcal P(\mu^2,\nu^2)\\
=\int_X\int_X k(x,y,dz) \Big(\mu^1(dx)(\nu^1-\nu^2)(dy)+\nu^2(dy)(\mu^1-\mu^2)(dx)\Big),
\end{multline*}
we finally obtain
\begin{multline*}
d\big(\mathcal P(\mu^1,\nu^1), \mathcal P(\mu^2,\nu^2)\big)\leq \int_X\int_X\bigg|\int_X \mathcal K(x,y,dz)\Psi(dy)\bigg|dz\mu^1(dx)\\
+\int_X\int_X\bigg|\int_X \mathcal K(x,y,dz)\Phi(dx)\bigg|dz\nu^2(dx)<\frac{ d(\nu^1, \nu^2) + d(\mu^1, \mu^2)}{2},
\end{multline*}
which implies (\ref{kontrakcja}).
\end{proof}

\begin{lem}\label{lemma2} Assume that $(\ref{kontrakcja})$ is satisfied for all $\mu^1,\mu^2,\nu^1,\nu^2\in\mathcal M_1$ such that $\int_X x \mu^1(dx)=\int_X x\mu^2(dx)$ and $\int_X x\nu^1(dx)=\int_X x\nu^2(dx).$ Fix $\mu^1_0,\mu^2_0,\nu^1_0,\nu^2_0\in\mathcal M_1$ satisfying $\int_X x \mu_0^1(dx)=\int_X x\mu^2_0(dx)$ and $\int_X x\nu^1_0(dx)=\int_X x\nu^2_0(dx).$ Denote by $(\mu^1_t,\nu^1_t)$ and $(\mu^2_t,\nu^2_t)$ solutions of system $(\ref{sysprobmesconst})$ with initial conditions $(\mu^1_0,\nu^1_0)$ and $(\mu^2_0,\nu^2_0)$, respectively. Then $\mu^1_t,\nu^1_t,\mu^2_t,\nu^2_t\in\mathcal M_{[\alpha,\beta]}$ for some $\alpha,\beta\in X$ and any $t>0$, and
\begin{equation}\label{zwezanie}
\max\big\{d(\nu^1_s, \nu^2_s), d(\mu^1_s, \mu^2_s)\big\}>\max\big\{d(\nu^1_t, \nu^2_t), d(\mu^1_t, \mu^2_t)\big\}
\end{equation}
for $0\leq s<t\leq T$ provided $\mu^1_T\neq \mu^2_T$ and $\nu^1_T\neq \nu^2_T.$ 
\end{lem}
\begin{proof}
Taking the mean value of both sides of equations in (\ref{sysprobmesconst}), from condition (\ref{sredniaK}) we obtain
\begin{equation}\label{meansys}
\left\{\begin{array}{l}m'(t)=\frac{1}{2}\big(n(t)-m(t)\big), \vspace*{0.2cm}
\\ n'(t)=\frac{A}{2}\big(m(t)-n(t)\big).\end{array}\right.
\end{equation}
where $m(t)=\int_X x\mu_t(dx)$ and  $n(t)=\int_X x\nu_t(dx)$. From above system we obtain $m(t)-n(t)=(m_0-n_0)e^{-(1+A)t/2}.$
Thus, if $m_0\geq n_0$ (resp. $m_0<n_0$), then from (\ref{meansys}) function $m(t)$ decreases (resp. increases) and $n(t)$ increases (resp. decreases), so $m(0)\geq m(t)\geq n(t)\geq n(0)$ (resp. $m(0)\leq m(t)\leq n(t)\leq n(0)$). Consequently, $\mu_t,\nu_t\in\mathcal M_{[\alpha,\beta]}$ for any $t>0$, where $\alpha:=\min\{m(0),n(0)\}$ and $\beta:=\max\{m(0),n(0)\}$.

Fix $s,t$ such that $0\leq s<t\leq T$. Every solution of system (\ref{sysprobmesconst}) is of the following form
\begin{equation}
\left\{\begin{array}{l}\mu_t=e^{s-t}\mu_s+\int_s^t e^{r-t}\mathcal P (\mu_r,\nu_r)dr, \vspace*{0.2cm}\\
\nu_t=e^{A(s-t)}\nu_s+A\int_s^t e^{A(r-t)}\mathcal P (\nu_r,\mu_r)dr.
\end{array}\right.
\end{equation}
From Lemma (\ref{lemma1}) we obtain
\begin{equation}\label{h1}
\left\{\begin{array}{l}e^t d(\mu^1_t,\mu^2_t)< e^{s}d(\mu^1_s,\mu^2_s)+\int_s^t e^{r}\max\big\{d(\nu^1_r,\nu^2_r),d(\mu^1_r,\mu^2_r)\big\}dr,\vspace*{0.2cm} \\
e^{At}d(\nu^1_t,\nu^2_t)< e^{As}d(\nu^1_s,\nu^2_s)+A\int_s^t e^{Ar}\max\big\{d(\nu^1_r,\nu^2_r),d(\mu^1_r,\mu^2_r)\big\}dr.
\end{array}\right.
\end{equation}
We divide the interval $[0,T]$ into subintervals $I_n:=[t_n,t_{n+1}]$ such that the sign of the difference $d(\mu^1_t,\mu^2_t)-d(\nu^1_t,\nu^2_t)$ is fixed for every $t\in I_n$. Consider any such interval $I_n$ and suppose for example that $d(\mu^1_t,\mu^2_t)\leq d(\nu^1_t,\nu^2_t)$ for every $t\in I_n$.
From the second inequality of (\ref{h1}) we obtain for $s,t\in I_n$
\begin{multline}
e^{At}\max\big\{d(\nu^1_t,\nu^2_t), d(\mu^1_t,\mu^2_t)\big\} \\
<e^{As}\max\big\{d(\nu^1_s,\nu^2_s), d(\mu^1_s,\mu^2_s)\big\}+A\int_s^t e^{Ar}\max\big\{d(\nu^1_r,\nu^2_r),d(\mu^1_r,\mu^2_r)\big\}dr,
\end{multline}
and from Gronwall inequality we obtain (\ref{zwezanie}) for $t,s\in I_n$. If in any interval the inequality is reversed, i.e. $d(\mu^1_t,\mu^2_t)\geq d(\nu^1_t,\nu^2_t)$ for $t\in I_n$, then we use the first inequality of (\ref{h1}) and again from Gronwall inequality we obtain (\ref{zwezanie}). Since intersection of intervals $I_n$ and $I_{n+1}$ is nonempty, (\ref{zwezanie}) holds for any $0\leq s<t\leq T$.
\end{proof}

\begin{lem}\label{lem5}
Assume that condition $(ii)$ from Theorem $\ref{mainthm}$ is satisfied. Then for every pair of measures $\mu_0,\nu_0\in\mathcal M_\gamma$ the orbits $\mathcal O(\mu_0), \mathcal O(\nu_0)$ are relatively compact subsets of $\mathcal M_{[\alpha,\beta]}$. Moreover, $\textup{cl}\,\mathcal O(\mu_0), \textup{cl}\,\mathcal O(\nu_0)\subset\mathcal M_{[\alpha,\beta]} \cap\mathcal M_\gamma,$ where $\textup{cl}$ is closure in $(\mathcal M_{[\alpha,\beta]},d).$
\end{lem}
\begin{proof}
Fix $\mu_0,\nu_0\in\mathcal M_\gamma$ such that $\int_X |x|^\gamma \mu_0(dx), \int_X |x|^\gamma \nu_o(dx)\leq m_0$ for some $m_0>0$. Since form Lemma \ref{lemma2} we obtain $\alpha\leq\int_Xx\mu_t(dx),\int_Xx\nu_t(dx)\leq \beta$ for some $\alpha,\beta\in X$ and every $t>0$, then the relative compactness of orbits follows from Lemma \ref{relcomp} 
provided we prove the following upper--bounds
\begin{equation}
\int_X |x|^\gamma \mu_t(dx), \int_X |x|^\gamma \nu_t(dx)\leq M,
\end{equation}
for some $M\geq m_0$ and any $t>0$. 
Notice that the set $Y:=C\Big([0,T], \mathcal{M}_{ [\alpha,\beta],\gamma, M}\times\mathcal{M}_{ [\alpha,\beta],\gamma, M}\Big)$ is a closed subset of $C\Big([0,T], \mathcal{M}_{[\alpha,\beta]}\times \mathcal{M}_{[\alpha,\beta]}\Big)$, and map $\Lambda(\mu,\nu)_t=\big(\Lambda^1(\mu,\nu)_t,\Lambda^2(\mu,\nu)_t\big)$, where
$\Lambda^1(\mu,\nu)_t=e^{-t}\mu_0+\int_0^t e^{r-t}\mathcal P (\mu_r,\nu_r)dr$ and 
$\Lambda^2(\mu,\nu)_t=e^{-At}\nu_0+A\int_0^t e^{A(r-t)}\mathcal P (\mu_r,\nu_r)dr$
is contraction for sufficiently small $T>0$, whose unique fixed point is $t\mapsto(\mu_t,\nu_t)$. We will show that the set $Y$ is invariant with respect to $\Lambda$, i.e., $\Lambda(Y)\subset Y$ for some constant $M>0$.
We calculate
\begin{multline*}
\int_X |x|^\gamma \Lambda^1(\mu,\nu)_t(dx)=e^{-t}\int_X |x|^\gamma\mu_0(dx)+\int_0^t e^{r-t}\int_X|x|^\gamma\mathcal P(\mu_r,\nu_r)dr\\
\leq e^{-t}\int_X |x|^\gamma\mu_0(dx)+\int_0^t e^{r-t}\bigg(C+L\max\bigg\{\int_X|x|^\gamma\mu_r(dx),\int_X|x|^\gamma\nu_r(dx)\bigg\}\bigg)dr\\
\leq e^{-t}\int_X |x|^\gamma\mu_0(dx)+\int_0^t e^{r-t}\Big(C+LM\Big)dr\leq M,
\end{multline*}
for $M$ such big that $C+LM\leq M.$ In the same way, we show that $\int_X |x|^\gamma \Lambda^2(\mu,\nu)_t(dx)\leq M$ for some constant $M>0$. Consequently, orbits are relatively compact.
\end{proof}

Consider family $\big(S(t)\big)_{t\geq 0}$ of transformations of $\mathcal M_{[\alpha,\beta]}\times\mathcal M_{[\alpha,\beta]}$ given by the formula $S(t)(\mu_0,\nu_0)=(S_1(t)\mu_0,S_2(t)\nu_0)=(\mu_t,\nu_t),$ where $(\mu_t,\nu_t)$ is the solution of system (\ref{sysprobmesconst}) with initial condition $(\mu_0,\nu_0)$. Consider $\omega$--limit set for $\mu,\nu\in\mathcal M_{[\alpha,\beta]},$ i.e.,
$$\omega(\mu,\nu)=\Big\{(\bar\mu,\bar\nu)\colon (\bar\mu,\bar\nu)=\lim_{n\to\infty} (\mu_{t_n},\nu_{t_n}) \textrm{ for some sequence } (t_n)_{n\in\N}\textrm{ s.t. }t_n\to\infty \Big\}.$$

\begin{proof}[Proof of Theorem \ref{mainthm}]
Take measures $\mu,\nu\in\mathcal M_{[\alpha,\beta]}\cap\mathcal M_{\gamma}$. From Lemma \ref{lem5} the orbits $\mathcal O(\mu),\mathcal O(\nu)$ are relatively compact in $\mathcal M_{[\alpha,\beta]}$. Consequently $\omega(\mu,\nu)$ is nonempty and compact set. Moreover, for $t>0$ $S(t)(\omega(\mu,\nu))=\omega(\mu,\nu).$ Suppose that $\omega(\mu,\nu)$ has more than one element. Then we can find $(\mu_1,\nu_1)$ and $(\mu_2,\nu_2)$ which maximize the function $\max\big\{d(\mu_1,\mu_2),d(\nu_1,\nu_2)\big\}.$ For any $t>0$ there exist $(\bar\mu_1,\bar\nu_1)$ and $(\bar\mu_2,\bar\nu_2)$ such that $S(t)(\bar\mu_1,\bar\nu_1)=(\mu_1,\nu_1)$ and $S(t)(\bar\mu_2,\bar\nu_2)=(\mu_2,\nu_2).$ From condition (i), Lemma \ref{lemma1} and Lemma \ref{lemma2} we obtain
\begin{multline}\label{zwww}
\max\big\{d(\mu_1,\mu_2),d(\nu_1,\nu_2)\big\}
=\max\big\{d\big(S_1(t)(\bar\mu_1,\bar\nu_1),S_1(t)(\bar\mu_2,\bar\nu_2)\big),\\
d\big(S_2(t)(\bar\mu_1,\bar\nu_1),S_2(t)(\bar\mu_2,\bar\nu_2)\big)\big\}<\max\big\{d(\bar \mu_1,\bar \mu_2),d(\bar \nu_1,\bar\nu_2)\big\}.
\end{multline}
Inequality (\ref{zwww}) contradicts the definition of $(\mu_1,\nu_1)$ and $(\mu_2,\nu_2)$. Hence $\omega(\mu,\nu)=\{(\mu^*,\nu^*)\},$ and $S(t)(\mu^*,\nu^*)=(\mu^*,\nu^*)$ for every $t>0$. Consequently,  $\mu^*=\mathcal P(\mu^*,\nu^*)=\nu^*$. According to Lemma \ref{lemma1} operator $\mathcal P$ has only one fixed point $(\mu^*,\mu^*),$ so the limit $\lim_{t\to\infty} S(t)(\mu,\nu)$ does not depend on $\mu,\nu\in\mathcal M_{[\alpha,\beta]}\cap\mathcal M_{\gamma}$. Consider now any measures $\mu,\nu\in\mathcal M_{[\alpha,\beta]}.$ Since the set $\mathcal M_{[\alpha,\beta]}\cap\mathcal M_{\gamma}$ is dense in $\mathcal M_{[\alpha,\beta]}$, for every $\varepsilon>0$ there exists $\bar\mu,\bar\nu\in\mathcal M_{[\alpha,\beta]}\cap\mathcal M_{\gamma}$ such that $d(\mu,\bar\mu),d(\nu,\bar\nu)<\varepsilon$. Since $\lim_{t\to\infty} S(t)(\bar \mu,\bar\nu)=(\mu^*,\mu^*)$ there exists $t_\varepsilon$ such that $d\big(S_1(t)(\bar\mu,\bar\nu),\mu^*\big)<\varepsilon$ and $d\big(S_2(t)(\bar\mu,\bar\nu),\mu^*\big)<\varepsilon$ for every $t>t_\varepsilon$. Since from Lemma \ref{lemma2} $S_1(t)$ are contractions we obtain
$$d\big(S_1(t)(\mu,\nu),\mu^*\big)\leq d\big(S_1(t)(\mu,\nu), S_1(t)(\bar\mu,\bar\nu)\big)+d\big(S_1(t)(\bar\mu,\bar\nu),\mu^*\big)<2\varepsilon,$$
and similarly $d\big(S_2(t)(\mu,\nu),\mu^*\big)<2\varepsilon$ for $t>t_\varepsilon$, which completes the proof.
\end{proof}
\subsection{Corollaries and further theorems on stability}
We start with investigation on the mean value of the limiting distribution $\mu^*$.
\begin{col} Under assumptions of Theorem~\ref{mainthm},
\begin{equation}\label{limitingmean}
\int_X x\mu^*(dx)=\frac{A\int_Xx\mu_0(dx)+\int_Xx\nu_0(dy)}{A+1}.
\end{equation}
\end{col}
\begin{proof}
Denote $\bar x:=\int_X x\mu^*(dx)$, $m(t):=\int_X x\mu_t(dx)$ and $n(t):=\int_X x\nu_t(dx)$. Since $\mu_t,\nu_t\to\mu^*$ and $\mu_t,\nu_t\in\mathcal M_{[\alpha,\beta]}$, we have also $m(t), n(t)\to \bar x$ as $t\to\infty$. From (\ref{meansys}) it follows that $\frac{d}{dt}\Big(Am(t)+n(t)\Big)=0.$ Consequently, $$Am(0)+n(0)=\lim_{t\to\infty}\Big(Am(t)+n(t)\Big)=(A+1)\bar x,$$
which gives formula (\ref{limitingmean}).
\end{proof}

Now we proceed to a result on asymptotic stability of solutions in stronger convergence. The proof of following result can be adapted from proof of Theorem~4 in \cite{rud_zwo}.
\begin{col}\label{strongconv}
Assume that the measure $k(x,y,dz)$ has bounded and continuous density with respect to Lebesgue measure, and suppose that assumptions of Theorem~\ref{mainthm} are satisfied. Then the stationary measure $\mu^*$ has continuous and bounded density $u^*$ with respect to Lebesgue measure. Moreover, for every $\mu_0,\nu_0\in\mathcal M_{[\alpha,\beta]}$, corresponding solutions $\mu_t,\nu_t$ of $(\ref{sysprobmesconst})$ can be written in the form 
$\mu_t=e^{-t}\mu_0+\bar \mu_t,$ $\nu_t=e^{-At}\nu_0+\bar \nu_t,$ where $\bar\mu_t, \bar\nu_t$ are absolute continuous measures with respect to Lebesgue measure, whose densities are continuous and bounded and converge to $u^*$ uniformly, as $t\to\infty$.
\end{col}

Until now, we studied asymptotic properties of simpler system (\ref{sysprobmesconst}) whose all coefficients are constant. Now we investigate asymptotic properties of solutions to the initial system (\ref{sysprobmes}) by comparing them with proper solutions of (\ref{sysprobmesconst}). Similar idea and techniques were previously used in \cite{rudmac}.

\begin{thm}\label{finalthm}
Suppose that $(\ref{persis})$ is satisfied and conditions \textup{(i)}, \textup{(ii)} of Theorem~\ref{mainthm} hold with $\alpha=\beta$. Then there exists $\mu^*\in\mathcal M_{1,\alpha}$ such that for any $\mu_0,\nu_0\in\mathcal M_{1,\alpha}$ coordinates of solution $(\mu_t,\nu_t)$ of system $(\ref{sysprobmes})$ with initial value $(\mu_0,\nu_0)$ converge to $\mu^*$ in $\mathcal M_{1,\alpha}$ as $t\to\infty$.
\end{thm}
\begin{proof}
Fix $s\geq 0$ and let $(\mu^1_t,\nu^1_t)$ and $(\mu^2_t,\nu^2_t)$ be solutions of (\ref{sysprobmesconst}) and (\ref{sysprobmes}) with the same initial condition $(\mu_s,\nu_s)$.
Then  $\int_X x\mu^1_t(dx)=\int_X x\mu^2_t(dx)$ and $\int_X x\nu^1_t(dx)=\int_X x\nu^2_t(dx)$ for any $t\geq s$, and consequently from (\ref{sysprobmes}), (\ref{sysprobmesconst}) and (\ref{kontrakcja}) it follows that
\begin{equation}\label{hh}
\left\{\begin{array}{l}e^t d(\mu^1_t,\mu^2_t)\!<\!e^s d(\mu^1_s,\mu^2_s)\!+\!\int_s^t e^{r}\max\big\{d(\nu^1_r,\nu^2_r),d(\mu^1_r,\mu^2_r)\big\}dr,\vspace*{0.2cm} \\
e^{At}d(\nu^1_t,\nu^2_t)\!<\!e^{As} d(\nu^1_s,\nu^2_s)\!+\!A\int_s^t e^{Ar}\max\big\{d(\nu^1_r,\nu^2_r),d(\mu^1_r,\mu^2_r)\big\}dr\!+\!e^{At}G(s,t),
\end{array}\right.
\end{equation}
where $G(s,t):=c\big|e^{-\int_s^t A(w)dw}-e^{A(s-t)}\big|+c\int_s^t \big|A(r)e^{-\int_r^t A(w)dw}-Ae^{A(r-t)}\big|dr$ and $c>0$ is a constant such that $d\big(\nu_r^2,0\big),  d\big(\mathcal P(\mu^2_r,\nu^2_r),0\big)<c$  for every $r>0$ (such constant $c$ exists, because $\nu_r^2,\mathcal P(\mu^2_r,\nu^2_r)\in\mathcal M_1$ for every $r>0$ due to assumption (\ref{ksr1})). 

We divide the interval $[0,\infty)$ into a sequence of subintevals $I_n$ of lengths $|I_n|\leq 1$ such that the sign of the difference $d(\mu^1_r,\mu^2_r)-d(\nu^1_r,\nu^2_r)$ is fixed for every $r\in I_n$. Fix $n\in\N$.  From appropriate inequality from (\ref{hh}), by Gronwall lemma we obtain for $s_n,t_n\in I_n$ 
\begin{equation}\label{gr}
\max\big\{d(\nu^1_{t_n},\nu^2_{t_n}),d(\mu^1_{t_n},\mu^2_{t_n})\big\}\!<\!\max\big\{d(\nu^1_{s_n},\nu^2_{s_n}),d(\mu^1_{s_n},\mu^2_{s_n})\big\}\!+\!H(s_n,t_n),
\end{equation}
where $H(s,t)=e^{A(t-s)}G(s,t)$. Notice that for any $s_n,t_n\in I_n,$ $s_n\leq t_n$
\begin{multline*}
H(s_n,t_n)\!\leq \!c\Big|e^{-\int_{s_n}^{t_n} \big(A(w)-A\big)dw}-1\Big|\!+\!c\int_{s_n}^{t_n} \Big|A(r)-A\Big|dr\!+\!cA\int_{s_n}^{t_n}\Big|e^{-\int_r^{t_n} \big(A(w)-A\big)dw}-1\Big|dr\\
\leq 2c \int_{s_n}^{t_n} \Big|A(w)-A\Big|dw+cA\int_{s_n}^{t_n}\int_r^{t_n}\Big|A(w)-A\big|dwdr\leq C\int_{s_n}^{t_n} \Big|A(w)-A\Big|dw
\end{multline*}
where $C=c\big(2+A\big)$. Since $\mu^1_s=\mu^2_s=\mu_s$ and $\nu^1_s=\nu^2_s=\nu_s$, from above inequality and (\ref{gr}) we obtain
\begin{equation}
\max\big\{d(\nu^1_t,\nu^2_t),d(\mu^1_t,\mu^2_t)\big\}<C\int_s^t \Big|A(w)-A\Big|dw.
\end{equation}
Notice that since $\big(M(t),F(t)\big)$ converges to $(\bar M,\bar F)$ exponentially, as $t\to\infty$, also $A(t)=M(t)/F(t)$ tends  to $A=\bar M/\bar F$ exponentially as well, i.e., there exists constants $a,b>0$ such that $|A(t)-A|\leq ae^{-bt}$ for all $t\geq 0$. 

Fix $\varepsilon>0$ and take $s>0$ such that $Ce^{-bs}/b<\varepsilon/4$. Let $(\mu_t,\nu_t)$ be solution of (\ref{sysprobmes}) with initial value $(\mu_0,\nu_0)$ and $(\bar \mu_t,\bar \nu_t)$ be solution of (\ref{sysprobmesconst}) such that  $(\bar \mu_s,\bar \nu_s)=(\mu_s,\nu_s)$.  Then for a large enough $t>s$ that $\max\big\{d(\mu^*,\bar\mu_t),d(\mu^*,\bar\nu_t)\big\}\leq \varepsilon/2$
\begin{multline*}
\max\big\{d(\nu_t,\mu^*),d(\mu_t,\mu^*)\big\}<\max\big\{d(\nu_t,\bar \nu_t),d(\mu_t,\bar \mu_t)\big\}+\max\big\{d(\mu^*,\bar\mu_t),d(\mu^*,\bar\nu_t)\big\}\\
\leq \frac{C}{b}\big(e^{-bs}-e^{-bt}\big)+\max\big\{d(\mu^*,\bar\mu_t),d(\mu^*,\bar\nu_t)\big\}\leq \varepsilon,
\end{multline*}
which completes the proof.
\end{proof}

Combining Corollary~\ref{strongconv} and Theorem~\ref{finalthm}, one can deduce the following

\begin{col}
Assume that the measure $k(x,y,dz)$ has bounded and continuous density with respect to Lebesgue measure and suppose that assumptions of Theorem~\ref{finalthm} are satisfied. If $u_0,v_0$ are bounded and continuous densities of initial measures $\mu_0,\nu_0$ of solutions $\mu_t,\nu_t$ to $(\ref{sysprobmes})$, then the densities of the measures $\mu_t,\nu_t$ converge uniformly to continuous and bounded density $u^*$ of stationary measure $\mu^*$, as $t\to\infty$.
\end{col}

\section{Examples}
Following examples come from considerations on hermaphroditic populations (see \cite{rud_zwo}), however they are also biologically reasonable for the two--sex populations case.
\subsection{Inheritance of mean parental trait with additive noise}
We suppose that $X=\R$. If $x,y\in \R$ are traits of parents, then we suppose that $\frac{x+y}{2}+Z$ is trait of their offspring, where $Z$ is zero--mean random variable distributed by some density $h$. We assume that $\E Z^2<\infty$ and $h(z)>0$ for all $z\in X$. Then the measure $k(x,y,dz)$ has following density
$$\kappa(x,y,z)=h\bigg(z-\frac{x+y}{2}\bigg).$$ It is easy to check that $\frac{\partial}{\partial x} \mathcal K(x,y,z)=-\frac{1}{2} h\Big(z-\frac{x+y}{2}\Big),$
and condition (i) from Theorem~\ref{mainthm} is satisfied if
$$\int_{-\infty}^\infty \big|h(z-a)-h(z-b)\big|dx<2$$
for all $a,b\in\R$. The above inequality is valid, since $h$ is probability density function, positive everywhere. 

Now we proceed to condition (ii). Fix two measures $\mu,\nu\in\mathcal M_{[\alpha,\beta]}$. Then
\begin{multline*}
\int_{-\infty}^\infty z^2\mathcal P(\mu,\nu)(dz)=\int_{-\infty}^\infty\int_{-\infty}^\infty\int_{-\infty}^\infty\bigg(\Big(z-\frac{x+y}{2}\Big)^2+\Big(z-\frac{x+y}{2}\Big)(x+y)+\frac{(x+y)^2}{4}\bigg)\\ \times h\bigg(z-\frac{x+y}{2}\bigg)\mu(dx)\nu(dy)dz\leq \E Z^2+\frac{1}{4}\int_{-\infty}^\infty\int_{-\infty}^\infty(x+y)^2\mu(dx)\nu(dy)\\
=\E Z^2 +\frac{1}{2}\bigg(\int_{-\infty}^\infty x\mu(dx)\bigg)\bigg(\int_{-\infty}^\infty x\nu(dx)\bigg)+\frac{1}{2}\max\bigg\{\int_{-\infty}^\infty x^2\mu(dx), \int_{-\infty}^\infty x^2\nu(dx)\bigg\}.
\end{multline*}
Since $\mu,\nu$ have their first moments upper--bounded by $\beta$, condition (ii) of Theorem~\ref{mainthm} is satisfied with $\gamma=2$, $C=\E Z^2+\beta^2/2$ and $L=\frac{1}{2}.$ From Theorem~\ref{mainthm} there exists unique limiting distribution $\mu^*\in\mathcal M_{[\alpha,\beta]}$ with first moment (\ref{limitingmean}) such that $\mu_t,\nu_t\to\mu^*$ in $\mathcal M_{[\alpha,\beta]}$. If additionally $\alpha=\beta$ and $h$ is bounded and continuous, limiting measure $\mu^*$ has continuous and bounded density $u^*$ and convergence of absolute continuous parts of $\mu_t,\nu_t$ is uniform by Corollary~\ref{strongconv}.

It turns out that, in this case of trait inheritance, it is possible to find the limiting distribution $\mu^*$ explicitly which has the form of infinite sequence of measure convolutions. In the case when $h$ has $0$--mean normal distribution with standard deviation $\sigma$, then the limiting distribution is also normal, with mean $\bar x=\alpha=\beta$ and standard deviation $\sqrt{2}\sigma$ (see \cite{rud_zwo}).

\subsection{Inheritance of mean parental trait with multiplicative noise} The following example is more reasonable for description of non--negative traits such as average body mass or height.
Thus we suppose that $X=[0,\infty)$. If $x,y$ are parental traits, then the trait of offspring is given by $(x+y)Z$, where $Z$ is $[0,1]$--valued random variable with mean $\frac{1}{2}$, distributed by density $h$. Then the density $\kappa(x,y,\cdot)$ of measure $k(x,y,dz)$ has the form
$$\kappa(x,y,z)=\frac{1}{x+y} h\bigg(\frac{z}{x+y}\bigg),$$
for $z\in[0,x+y]$, $x+y>0$ or $\kappa(x,y,z)=0$ otherwise. Assume that there exists $\varepsilon>0$ such that support of the function $h$ contains $(0,\varepsilon)$. One can easily compute
$$\frac{\partial}{\partial x} \mathcal K(x,y,z)=-h\bigg(\frac{z}{x+y}\bigg) \frac{z}{(x+y)^2.}$$
The condition (i) of Theorem~\ref{mainthm} is equivalent to 
$$\int_0^\infty \Big|h\Big(\frac{z}{a}\Big)\frac{z}{a^2}-h\Big(\frac{z}{b}\Big)\frac{z}{b^2}\Big|dz<1,$$
for all $a,b\in[0,\infty)$. Above inequality is satisfied, since the function $h$ has mean equal to $\frac{1}{2}$, and the interval $(0,\varepsilon)$ is in its support. 

We check condition (ii). Take $\mu,\nu\in\mathcal M_{[\alpha,\beta]}$. Then
\begin{multline*}
\int_0^\infty z^2\mathcal P(\mu,\nu)dz=\E Z^2 \int_0^\infty\int_0^\infty (x+y)^2\mu(dx)\nu(dy)\leq \\
 2\E Z^2\bigg(\int_{-\infty}^\infty x\mu(dx)\bigg)\bigg(\int_{-\infty}^\infty x\nu(dx)\bigg)+2\E Z^2\max\bigg\{\int_{-\infty}^\infty x^2\mu(dx), \int_{-\infty}^\infty x^2\nu(dx)\bigg\}.
\end{multline*}
Since the first moments of the measures $\mu,\nu$ are bounded by $\beta$ and $2\E Z^2 <2\E Z=1$, condition (ii) is satisfied with $\gamma=2$, $C=2\beta^2 \E Z^2$ and $L=2\E Z^2$.
\section{Conclusions}
In the paper we introduced some individual--based model in order to describe the evolution of non--sex--liked phenotypic traits in two--sex populations. The model includes semi--random mating of individuals of the opposite sex, natural death and intra--specific competition. Having passed the number of individuals to infinity, we derived the macroscopic system of equations for evolution of trait distributions. The main results of the investigation on solutions of the system are: existence and uniqueness of solutions in space of measures, study of total number of individuals in order to derive criteria for persistence or extinction, formulation of the conditions implying existence of the unique stable distribution and its asymptotic stability. Moreover, under additional assumptions, we studied the existence and asymptotic properties of solutions from the standpoint of their densities.

Now we interpret some of the results in biological language. We start with inequalities (\ref{persis}) and (\ref{extinct}). If we consider the numbers $\alpha_m=p_m/D_m$ and $\alpha_f=p_f/D_f$ as an environmental adaptation of males and females, then  (\ref{persis}) implies that the mean environmental adaptation in population is greater than one. On average, one dying individual is replaced by more than one newborn, and consequently whole the population persists. The opposite inequality (\ref{extinct}) leads to extinction of both subpopulations. A possible scenario assumes that one of the numbers $\alpha_m,\alpha_f$ is strictly smaller than $1$. Despite this, it is still possible that (\ref{persis}) holds. It means, that the whole population can survive, although one of the subpopulations has smaller mating rate. In particular, if we take $p_m=0$, then populations with minor male mating rates are also covered by our model. In that case the growth of the population depends on female mating and death rates, and inequality $p_f/D_f>2$ means that the population persists.

The existence and asymptotic stability of the stable distribution $\mu^*$, and the fact that this distribution is the same for both male and female populations is an intuitive consequence of Fisher's principle and inheritance of traits which are non--sex--linked. The result suggests that after a long time two--sex populations behave as they were hermaphroditic, provided we investigate only the evolution of non--sex--linked traits. This result enlarges the area of applications of the hermaphroditic model derived and studied in \cite{rud_zwo} also to two--sex populations.

The future perspectives for the model presented in the paper are multidirectional. The most interesting issue would be to study analogous model including assortative mating of individuals with aid of trait--dependent marriage functions instead of semi--random mating (see e.g. \cite{yang}). Marriage functions reflect preferences for possible partners in the population and influence the shape of the trait distribution. The long--time behavior of the corresponding macroscopic equations could answer how big this influence is in stable population, and what shape of limiting distribution we should expect. Another issue is to study how trait values of individuals affect their fitness. Since in general our model allows to include trait--dependent rates, it would be interesting to study long--time behavior of corresponding solutions in case when there are two different fitness optima for males and females (see e.g. \cite{bond}).

\end{document}